\documentclass{article}
\usepackage{graphicx} % Required for inserting images
\usepackage{amsmath}
\usepackage{amsthm}
\usepackage{amssymb}
\usepackage{biblatex}
\begin{document}
\title{The super approximation property of $\mathrm{SL}_2(\mathbb{Z}/q\mathbb{Z}) \times \mathrm{SL}_2(\mathbb{Z}/q\mathbb{Z}) \times \mathrm{SL}_2(\mathbb{Z}/q\mathbb{Z})$}
\author{Chong Zhang}
\date{}
\maketitle

\begin{abstract}
Take $S \subset \mathrm{SL}_2(\mathbb{Z}) \times \mathrm{SL}_2(\mathbb{Z})\times \mathrm{SL}_2(\mathbb{Z})$ be finite symmetric and assume $S$ generates a group $G$ which is Zariski-dense in $\mathrm{SL}_2 \times \mathrm{SL}_2\times \mathrm{SL}_2(\mathbb{Z})$. This paper proves that the Cayley graphs
$$
\{\mathcal{C} a y(G(\bmod q), S(\bmod q))\}_{q \in \mathbb{Z}_{+}}
$$
form a family of expanders.
\end{abstract}
\newtheorem{Conjecture}{Conjecture}[equation]
\newtheorem{theorem}[equation]{Theorem}
\newtheorem{lemma}[equation]{Lemma}
\newtheorem{proposition}[equation]{Proposition}
\numberwithin{equation}{section}
\maketitle
\section{Introduction}
Consider a symmetric set $S$ which generate a subgroup $G=\langle S\rangle$ of $\mathrm{GL}_n(\mathbb{Z})$. We can define the super approximation property of the $G$. For any positive integer $q$, let $G_q=G(\bmod q)$ and $h_q$ be the Cheeger constant of the Cayley graph $\mathcal{C} a y\left(G_q, S(\bmod q)\right)$ given by
$$
h_q:=\min \left\{\frac{|\partial A|}{|A|}: A \subset G_q, 0<A \leq \frac{1}{2}\left|G_q\right|\right\},
$$
where $\partial A$ is the set of edges in $\mathcal{C} a y\left(G_q, S(\bmod q)\right)$ connecting one vertex in $A$ and one vertex in $G_q-A$. The super
approximation property of $G$ related to a set $\mathcal{A}$ of positive integers if there is $\epsilon>0$ such that $h_q>\epsilon, \forall q \in \mathcal{A}$. From previous work we know that this property is independent of the choice of the finite generating set $S$. We call $G$ is super approximation if $\mathcal{A}=\mathcal{Z}$.\\
We have already known that a plenty of lattices in $\mathrm{GL}_n(\mathbb{Z})$ satisfy the super approximation property due to the work of Margulis \cite{Mar73}. In 2008, Bourgain and Gamburd developed an analytic-combinatorial tool which is now called the "Bourgain-Gamburd expansion machine", which allows them to prove the super approximation property for any Zariski-dense subgroup of $\mathrm{SL}_2(\mathbb{Z})$ with respect to prime moduli \cite{BG08b}, another critical ingredient is Helfgott's triple product theorem \cite{Hel08}. Since then, there has been a series of papers extending Bourgain-Gamburd's Theorem to more general groups with respect to more general moduli \cite{BG08a}, \cite{BG09}, \cite{BGS10}, \cite{GV12}, \cite{Var12}, \cite{BV12}, \cite{Gol19}.\\
In \cite{GV12},there is a conjecture:
\begin{Conjecture} (Question 2, \cite{GV12}). Let $G<S L_n(\mathbb{Z})$ be finitely generated, then $G$ has the super approximation property with respect to all positive integers if and only if the identity component $\mathbb{G}_0$ of the Zariski closure $\mathbb{G}$ of $G$ is perfect, i.e. $\left[\mathbb{G}_0,\mathbb{G}_0\right]=\mathbb{G}_0$.
\end{Conjecture}
In the same paper \cite{GV12} and following development generalized by Salehi-Golsefidy to bounded powers of square free numbers.\\
\begin{theorem}
(Salehi-Golsefidy \cite{Gol19}) Let $G<S L_n(\mathbb{Z})$ be finitely generated, then $G$ has the super approximation property with respect to bounded powers of square free integers if and only if the identity component $\mathbb{G}_0$ of the Zariski closure $\mathbb{G}$ of $G$ is perfect, i.e. $\left[\mathbb{G}_0, \mathbb{G}_0\right]=\mathbb{G}_0$.    
\end{theorem}

There are still many technical challenges to generalized this conjecture, In many aforementioned works, the sum-product theorem is 
an efficient method, in this paper, we follow the step in \cite{tang2023super} to 
prove our main theorem of this paper:

\begin{theorem}
Let $S \subset \mathrm{SL}_2(\mathbb{Z}) \times \mathrm{SL}_2(\mathbb{Z}) \times \mathrm{SL}_2(\mathbb{Z})$ be a finite symmetric set, and assume that it generates a group $G$ which is Zariski-dense in $\mathrm{SL}_2 \times \mathrm{SL}_2 \times \mathrm{SL}_2$, then $G$ has the super approximation property with respect to all positive integers.
\end{theorem}
We will use Theorem $\mathbf{1.1}$ in the special case $\mathbb{G}_0=\mathrm{SL}_2 \times \mathrm{SL}_2 \times \mathrm{SL}_2$ as a blackbox.\\\\

\noindent\textbf{Acknowledgements}:Thank Tang jincheng,Zhang Xin,Xiao xuanxuan for precious suggestion on this paper.

\section{Notation}
In this part we will present some notations which will be used in the paper.
The multiplicative unit of any written group is denoted by 1. Besides, if there is a ring structure, we denote the additive unit by 0.For any two subsets $A$ and $B$, we define their product set by $A \cdot B=\{a b \mid a \in A, b \in B\}$, and their sum set by $A+B=$ $\{a+b \mid a \in A, b \in B\}$. The $k$-fold product of $A$ is denoted by $A^k$, and the $k$-fold sum of $A$ is denoted by $\sum_k A$.

For any two complex valued functions $f$ and $g$ on a discrete group $G$, we define  $f * g$  as their convolution
$$
f * g(x)=\sum_{y \in G} f(y) g\left(x y^{-1}\right) .
$$

We write $f^{(l)}$ for the $k$-fold convolution of $f$ with itself.\\

Now we define the division $\|$,For a prime $p$, consider $p^n \| q$ if $p^n \mid q$ but $p^{n+1} \nmid q$ and for two integers $q_1$ and $q_2$, we write $q_1 \| q_2$ if for every $p^n \| q_1$, we also have $p^n \| q_2$. This definition can be extend to any ideal in a similar way.

For $q=\prod_i p_i^{n_i} \in \mathbb{Z}_{+}$and $\alpha$ a real positive number, we let $q^{\{\alpha\}}=\prod_i p_i^{\left[n_i \alpha\right]}$, where $\left[n_i \alpha\right]$ is the integer part of $n_i \alpha$.

Let $\pi_q: \mathbb{Z} \rightarrow q \mathbb{Z}$ be the residue map, which induces residue maps in various other contexts, and we denote them by $\pi_q$ as well.

Let $\Gamma=\mathrm{SL}_2(\mathbb{Z}) \times \mathrm{SL}_2(\mathbb{Z}) \times \mathrm{SL}_2(\mathbb{Z})$ and $\Lambda=\mathrm{SL}_2(\mathbb{Z}), \Gamma_q=\Gamma(\bmod q)$ and $\Gamma(q)$ be the kernel of the reduction map $\pi_q: \Gamma \rightarrow \Gamma_q$. Same meanings for $\Lambda(q)$ and $\Lambda_q$. We denote by $\mathbb{P}_i(i=1,2,3)$ the projection of $\Lambda \times \Lambda \times \Lambda$ to its $i^{\text {th}}$ factor. Sometimes we need to reduce the three factors of $\Lambda \times \Lambda \times \Lambda$ by three different moduli and we denote by $\pi_{q_1, q_2,q_3}$ the reduction map $\Lambda \times \Lambda \times \Lambda \rightarrow \Lambda_{q_1} \times \Lambda_{q_2}\times \Lambda_{q_3}$.

We adopt $f(q)<q^{c+}$to mean $f(q)<q^{c+\epsilon}$ for arbitrarily small $\epsilon$ when $q$ large. Similarly, $f(q)>q^{c-}$ to mean $f(q)>q^{c-\epsilon}$ for arbitrarily small $\epsilon$ when $q$ large.

Let $\chi_s$ be the normalized uniform counting measure supported on $S$, i.e., for $A \subset \Lambda \times \Lambda \times \Lambda$, $\chi_S (A)=\frac{|A \cap S|}{|S|}$. Let $\left.\pi_q[\chi s\right]$ be the pushforward of $\chi s$ under the residue map $\pi_q$. Let $T_q$ be the convolution operator by $\pi_q[\chi s]$, i.e., For $f \in l^2\left(G_q\right)$,
\begin{equation}
T_q(f)=\pi_q[\chi s] * f \text {. }    
\end{equation}

Then $T_q$ is a self adjoint operator on $l^2\left(G_q\right)$ with an invariant subspace $l_0^2\left(G_q\right)$ consisting of functions with average 0 . Denote the set of eigenvalues of $T_q$ on $l_0^2\left(G_q\right)$ by $E_q$.The previous work from Alon and Milman\cite{AM85} and \cite{BV12} make the theorem into the following proposition.

\begin{proposition}
Let $S \subset \Gamma$ be symmetric, and assume that it generates a group $G$ which is Zariski-dense in $\mathrm{SL}_2 \times \mathrm{SL}_2 \times \mathrm{SL}_2$. Then for any $\epsilon>0$ there is a $\delta>0$ such that the following holds. If $q \in \mathbb{Z}_{+}$sufficiently large, $A \subset \Gamma$ symmetric, and some integer $l$ satisfying
\begin{equation}
\chi_S^{(l)}(A)>q^{-\delta}, \quad l>\delta^{-1} \log q \quad \text { and } \quad\left|\pi_q(A)\right|<\left|\Gamma_q\right|^{1-\epsilon},    
\end{equation}
then
\begin{equation}
\left|\pi_q(A \cdot A \cdot A)\right|>\left|\pi_q(A)\right|^{1+\delta}
\end{equation}   
\end{proposition}
In the rest of the paper we will focus on proving Proposition 2.2. We give a sketch here.
\subsection{Sketch of the proposition 2.2}
Considering $q=q_s q_l$, where $q_s$ is the product of all exact prime power divisors of $q$ with small exponents, and $q_l=q / q_s$.

If $q$ is very small, then Theorem $\mathbf{1}$ is trivially hold because the condition is void.

If $q_l$ is not particularly small, we assume that a set A satisfies all the assumptions of proposition 2.2 but fails to satisfy the conclusion of proposition 2.2 which can derive a contradiction. We first show that there exists a large exact divisor $q^{\prime}$ of $q$ and some constant $C$ such that $\mathbb{P}_1\left(A^C\right)$,$\mathbb{P}_2\left(A^C\right)$ or $\mathbb{P}_3\left(A^C\right)$ contains a large congruence subgroup of $\Lambda_{q^{\prime}}$.For this part, we can refer the section 5 proposed in \cite{tang2023super} directly with minor change.

On the second stage, we glue local pieces together. The Proposition 3.1 of \cite{tang2023super} is the main gluing tool.To describe the idea,let us suppose a set $B \subset \Lambda_{q_1} \times \Lambda_{q_2} \times \Lambda_{q_3}$ such that $\mathbb{P}_1(B), \mathbb{P}_2(B),\mathbb{P}_3(B)$ are very large subsets of $\Lambda_{q_1}, \Lambda_{q_2},\Lambda_{q_3}$ for three exact and not necessarily coprime divisors $q_1, q_2,q_3$ of $q$. Let's just say $\mathbb{P}_1(B)=\Lambda_{q_1}$ , $\mathbb{P}_2(B)=\Lambda_{q_2}$ and $\mathbb{P}_3(B)=\Lambda_{q_3}$. We claim that we can find a large exact divisor $q_4$ of $q_2$ or $q_3$ such that a product set of $A \cup B$ can cover a very large subset of $\Lambda_{q_1} \times \Lambda_{q_4}\times \Lambda_{q_3}$ or $\Lambda_{q_1} \times \Lambda_{q_2}\times \Lambda_{q_4}$ modulo $\left(q_1,q_4, q_3\right)$ or $\left(q_1,q_2, q_4\right)$.To simplified the process, we discuss the divisor $q_4$ of $q_2$.For this, we consider a connecting map
$$
\psi: \Lambda_{q_1} \rightarrow B
$$
such that $\mathbb{P}_1 \circ \psi$ is identity.Consider a parameter $0<\theta<1$. By the Proposition 3.1 of \cite{tang2023super}, We can divide the possible situation into two scenarios:\\
(1) There exists a large exact divisor $q_4$ of $q_2$, and $x, y \in \Lambda_{q_1}$ such that for any $p^n \| q_4$, $\psi(x y) \neq \psi(x) \psi(y)\left(\bmod p^{\left[n \theta\right]}\right)$.\\
(2) There exists a large exact divisor $\bar{q}_4$ of $q_2$, and a large subset $S$ of $\Lambda_{q_1}$ such that for any $x, y \in S$, any $p^n \| q_4, \psi(x y)=\psi(x) \psi(y)\left(\bmod p^{[n \theta]}\right)$.\\
In the first scenario, we can conjugate the element $\psi(x y) \psi(x)^{-1} \psi(y)^{-1}$ by $B$ to create a large subset with first component 1 ,which is easy to follow.\\
In the second scenario, we consider two subcases. \\
The first one is that if there exists a large exact divisor $q_4$ of $\bar{q}_4$ such that $\pi_{q_4^{\{\theta\}}} \circ \mathbb{P}_1 \circ \psi=1$, then by taking commutator of $\psi\left(\Lambda_{q_1}\right)$ Iteratively, one can reconstruct a large set with a second component of 1, leading to the claim.

Another condition is that for a large exact divisor $q_4$ of $\tilde{q}_4$ such that for any $p^n \| q_4$, we have $\pi_{q_4^{\{(\frac{\theta}{2})}} \circ \mathbb{P}_1 \circ \varphi \neq 1$, then a one-parameter group $\mathcal{P}$ can be construct to extend the set so that we can capture all the prime divisors of $q_1$ and $q_4$. 

We can apply the claim we proposed to create a large subset of $\Gamma_q$ if the claim is true, by this subset we can ensure the proposition 2.2 is hold.

\section{Preliminary of combinatorics}
We proposed the following bounded generation result over $SL_2($Z$)\times SL_2($Z$)\times SL_2($Z$)$.\
\begin{proposition}
For any $0<\delta<\frac{1}{40}$,there exists $\varepsilon=\varepsilon(\delta)>0$ and an absolute constant $k\in N$ if the following condition satisfied:\\       
Let $A\subset \mathrm{SL}_2(\mathbb{Z}/q_1\mathbb{Z}) \times \mathrm{SL}_2(\mathbb{Z}/q_2\mathbb{Z}) \times \mathrm{SL}_2(\mathbb{Z}/q_3\mathbb{Z})$ be symmetric and $\lvert A \rvert>(q_1q_2q_3)^{3-\delta}$. Then there exists  $q_1^{\prime}\|q_1$,$q_2^{\prime}\| q_2$,$q_3^{\prime}\| q_3$,$q_1^{\prime}q_2^{\prime}q_3^{\prime}<(q_1q_2q_3)^{96\delta}$ so that
$$
A^{7200}\supset \mathrm{SL}_2(q_1^{\prime}\mathbb{Z}/q_1\mathbb{Z}) \times \mathrm{SL}_2(q_2^{\prime}\mathbb{Z}/q_2\mathbb{Z}) \times \mathrm{SL}_2(q_3^{\prime}\mathbb{Z}/q_3\mathbb{Z})
$$  
\end{proposition} 
To proposed this proposition we adopt ideas from Helfgott's approaches which solve the generation problem of $\mathrm{SL}_2(F_p)$ and some lemmas proved in \cite{tang2023super}.\\

\begin{lemma}
    Let $0<\delta<\frac{1}{20}$ and let $A,B\in \mathbb{Z}/q_1\mathbb{Z}\times \mathbb{Z}/q_2\mathbb{Z}\times \mathbb{Z}/q_3\mathbb{Z}$ such that $\lvert A \rvert,\lvert B \rvert>(q_1q_2q_3)^{1-\delta}$. Then we can obtain $q_1^{'}\lvert q_1$,$q_2^{'}\lvert  q_2$,$q_3^{'}\lvert q_3$,$q_1^{'}q_2^{'}q_3^{'}<(q_1q_2q_3)^{24\delta}$ so that
    \begin{equation}
    \sum_{240}AB-AB \supset q_1^{\prime}\mathbb{Z}/q_1\mathbb{Z}\times q_2^{\prime}\mathbb{Z}/q_2\mathbb{Z}\times q_3^{\prime}\mathbb{Z}/q_3\mathbb{Z}   
    \end{equation}
\end{lemma}
\begin{proof}
Let $P_i,i=1,2,3$ be the projection from $\mathbb{Z}/q_1\mathbb{Z}\times \mathbb{Z}/q_2\mathbb{Z}\times \mathbb{Z}/q_3\mathbb{Z}$ to the $i$-th component.Without loss of generality, We can assume $q_1=(q_2q_3)^\alpha$,$0<\alpha \leq 1$. Since $\lvert A \rvert>(q_1q_2q_3)^{1-\delta}$,we can define a subset of A with a fixed $x_0\in A$ satisfied the following condition:
$$
\lvert\{x\in A:P_1(x)=x_0\}\rvert
> (q_2q_3)^{1-\delta-\alpha\delta}
$$
which implies a subset $A^{\prime}\in A-A$,$\lvert A^{\prime} \rvert> (q_2q_3)^{1-\delta-\alpha\delta}$,$P_1(A^{\prime})=0$,similarly, there is $B^{\prime}\in B-B$,$\lvert B^{\prime} \rvert> (q_2q_3)^{1-\delta-\alpha\delta}$,$P_1(B^{\prime})=0$, apply the lemma 3.15 proved in \cite{tang2023super}, we obtain $q_2^{\prime}\lvert  q_2$,$q_3^{\prime}\lvert q_3$,$(q_2^{\prime}q_3^{\prime})>q_2q_3^{10(\delta+\alpha\delta)}$ such that
\begin{equation}
\sum_{192}AB-AB \supset \sum_{96}A^{\prime}B^{\prime}-A^{\prime}B^{\prime}\supset \mathbb{Z}/q_1\mathbb{Z}\times \mathbb{Z}/q_2\mathbb{Z}\times \mathbb{Z}/q_3\mathbb{Z}    
\end{equation}
If $\alpha<5\delta$, we can just take $q_1^{\prime}=q_1$.If $\alpha>5\delta$,then there are $A^{\prime\prime}$ and $B^{\prime\prime}$,$\lvert A^{\prime\prime}\rvert>q_1^{1-\delta-\frac{\delta}{\alpha}}$,$\lvert B^{\prime\prime} \rvert>q_1^{1-\delta-\frac{\delta}{\alpha}}$,$P_2(B^{'})=P_3(B^{\prime})=P_2(A^{\prime})=P_3(A^{\prime})=0$,the exponent $1-\delta-\frac{\delta}{\alpha}$ exceed $\frac{3}{4}$,so we applying the lemma 3.13 in \cite{tang2023super}. we obtain $q_1^{\prime}\lvert q_1$,$q_1^{\prime}<q_1^{\frac{12(\alpha+\frac{\delta}{\alpha}}{5})}$ so that
\begin{equation}
\sum_{48}AB-AB \supset \sum_{24}A^{\prime}B^{\prime}-A^{\prime}B^{\prime}\supset q_1^{\prime}Z/q_1{Z}\times q_2Z/q_2{Z}\times q_3Z/q_3{Z}    
\end{equation}
Adding (3.4) and (3.5), we obtain  with 
\[
    q_1^{\prime}q_2^{\prime}q_3^{\prime}<q_1^{\frac{12(\alpha+\frac{\delta}{\alpha}}{5})}q_2q_3^{10(\delta+\alpha\delta)}<(q_1q_2q_3)^{24\delta}
\]
\end{proof}
\begin{proof}
Since $\lvert A \rvert>(q_1q_2q_3)^{3-\delta}$,by the pigeon holes principle, we can obtain a set 
\begin{flalign*}
&\{(\gamma_1,\gamma_2,\gamma_3):\gamma_1 \text{has lower row} \Vec{u}, \gamma_2 \text{has lower row} \Vec{v},\gamma_3 \text{has lower row} \Vec{w}\}    
\end{flalign*}
Where $\Vec{u},\Vec{v},\Vec{w}$ satisfied
\begin{align*}
    \Vec{u}\in (Z/q_1{Z})^2\\
    \Vec{v}\in (Z/q_2{Z})^2\\
    \Vec{w}\in (Z/q_3{Z})^2
\end{align*}
whose cardinal exceeds $(q_1q_2q_3)^{1-\delta}$, so the cardinal of the set
$$
A_1:=A \cdot A^{-1} \cap\left\{\left(\left(\begin{array}{cc}
1 & m \\
0 & 1
\end{array}\right),\left(\begin{array}{cc}
1 & n \\
0 & 1
\end{array}\right),\left(\begin{array}{cc}
1 & t \\
0 & 1
\end{array}\right)\right):\right\}
$$
where
\begin{align*}
&m \in \mathbb{Z} / q_1 \mathbb{Z}\\&n \in \mathbb{Z} / q_2 \mathbb{Z}\\&t\in \mathbb{Z} / q_3 \mathbb{Z}    
\end{align*}
exceeds $(q_1q_2q_3)^{1-\delta}$.
Similarly, we can get the set
$$
A_1:=A \cdot A^{-1} \cap\left\{\left(\left(\begin{array}{cc}
1 & 0 \\
m & 1
\end{array}\right),\left(\begin{array}{cc}
1 & 0 \\
n & 1
\end{array}\right),\left(\begin{array}{cc}
1 & 0 \\
t & 1
\end{array}\right)\right)\right\}
$$
where
\begin{align*}
&m \in \mathbb{Z} / q_1 \mathbb{Z}\\&n \in \mathbb{Z} / q_2 \mathbb{Z}\\&t\in \mathbb{Z} / q_3 \mathbb{Z}    
\end{align*}
Define an equivalence relation $\sim$ on   $\mathrm{SL}_2(\mathbb{Z}/q_1\mathbb{Z}) \times \mathrm{SL}_2(\mathbb{Z}/q_2\mathbb{Z}) \times \mathrm{SL}_2(\mathbb{Z}/q_3\mathbb{Z})$ as
\[
(\gamma_1,\gamma_2,\gamma_3) \backsim (\gamma_1^{\prime},\gamma_2^{\prime},\gamma_3^{\prime})
\]
if and only if the second rows of each $\gamma_i,\gamma_i^{\prime}$, $i=1,2,3$ are up to 
a scalar in $(Z/q_i{Z})^{\ast}$,$i=1,2,3$,there are at most $q_1q_2q_3$ many such classes.So by the pigeon hole principle,  we can deduce the existence of a class containing at least $(q_1q_2q_3)^{2-\delta}$ many elements from $A$.
This implies the set
\begin{align*}
H_0=A \cdot A^{-1} \cap\left\{\left(\left(\begin{array}{cc}
\lambda_1 & x \\
0 & \lambda_1^{-1}
\end{array}\right),\left(\begin{array}{cc}
\lambda_2 & y \\
0 & \lambda_2^{-1}
\end{array}\right),\left(\begin{array}{cc}
\lambda_3 & z \\
0 & \lambda_3^{-1}
\end{array}\right)\right)\right\}    
\end{align*}
where satisfied the following condition
\begin{align*}
&\lambda_1 \in\left(\mathbb{Z} / q_1 \mathbb{Z}\right)^*, x \in \mathbb{Z} / q_1 \mathbb{Z},\\&\lambda_2 \in\left(\mathbb{Z} / q_2 \mathbb{Z}\right)^*, y \in \mathbb{Z} / q_2 \mathbb{Z},\\&\lambda_3 \in\left(\mathbb{Z} / q_3 \mathbb{Z}\right)^*, z \in \mathbb{Z} / q_3 \mathbb{Z}   
\end{align*}
has cardinal exceed $(q_1q_2q_3)^{2-\delta}$,
By the pigeon hole principle,there are $x_0\in Z/q_1{Z},y_0\in Z/q_2{Z},z_0\in Z/q_3{Z}$ so that
$$
H_0=A \cdot A^{-1} \cap\left\{\left(\left(\begin{array}{cc}
\lambda_1 & 0 \\
x_0 & \lambda_1^{-1}
\end{array}\right),\left(\begin{array}{cc}
\lambda_2 & 0 \\
y_0 & \lambda_2^{-1}
\end{array}\right),\left(\begin{array}{cc}
\lambda_3 & 0 \\
z_0 & \lambda_3^{-1}
\end{array}\right)\right)\right\}
$$
where the parameters satisfied:
\begin{align*}
&\lambda_1 \in\left(\mathbb{Z} / q_1 \mathbb{Z}\right)^*, x_0 \in \mathbb{Z} / q_1 \mathbb{Z},\\&\lambda_2 \in\left(\mathbb{Z} / q_2 \mathbb{Z}\right)^*, y_0 \in \mathbb{Z} / q_2 \mathbb{Z},\\&\lambda_3 \in\left(\mathbb{Z} / q_3 \mathbb{Z}\right)^*, z_0 \in \mathbb{Z} / q_3 \mathbb{Z}   
\end{align*}
has cardinal  $>(q_1q_2q_3)^{1-\delta}$. 
It's easy to compute the following result
\[
\begin{pmatrix}
    \lambda_1 & x_0 \\ 0 & \lambda_1^{-1}
\end{pmatrix} \begin{pmatrix}
    1 & m \\ 0 & 1
\end{pmatrix} \begin{pmatrix}
    \lambda_1 & x_0 \\ 0 & \lambda_1^{-1}
\end{pmatrix}^{-1} = \begin{pmatrix}
    1 & \lambda_1^{2}m \\ 0 & 1
\end{pmatrix}
\]
Similarly,
\[
\begin{pmatrix}
    \lambda_2 & y_0 \\ 0 & \lambda_2^{-1}
\end{pmatrix} \begin{pmatrix}
    1 & n \\ 0 & 1
\end{pmatrix} \begin{pmatrix}
    \lambda_2 & y_0 \\ 0 & \lambda_2^{-1}
\end{pmatrix}^{-1} = \begin{pmatrix}
    1 & \lambda_2^{2}n \\ 0 & 1
\end{pmatrix}
\]
\[
\begin{pmatrix}
    \lambda_3 & z_0 \\ 0 & \lambda_3^{-1}
\end{pmatrix} \begin{pmatrix}
    1 & t \\ 0 & 1
\end{pmatrix} \begin{pmatrix}
    \lambda_3 & z_0 \\ 0 & \lambda_3^{-1}
\end{pmatrix}^{-1} = \begin{pmatrix}
    1 & \lambda_3^{2}t \\ 0 & 1
\end{pmatrix}
\]
Applying lemma to the set
$$
\left\{\left(\lambda_1^{2},\lambda_2^{2},\lambda_3^{2}\right):\left(\begin{pmatrix}
    \lambda_1 & x \\ 0 & \lambda_1^{-1}
\end{pmatrix},\begin{pmatrix}
    \lambda_2 & y \\ 0 & \lambda_2^{-1}
\end{pmatrix},\begin{pmatrix}
    \lambda_3 & z \\ 0 & \lambda_3^{-1}
\end{pmatrix}\right) \in H \right\}
$$
and
$$
\left\{\left(m,n,t\right):\left(\begin{pmatrix}
    1 & m \\ 0 & 1
\end{pmatrix},\begin{pmatrix}
    1 & n \\ 0 & 1
\end{pmatrix},\begin{pmatrix}
    1 & t \\ 0 & 1
\end{pmatrix}\right) \in A_1 \right\}
$$
with exponent $1-2\delta$,then we have $Q_1\lvert q_1$,$Q_2\lvert q_2$,$Q_3\lvert q_3$,$Q_1Q_2Q_3<(q_1q_2q_3)^{48\delta}$ and
$$
\left(A \cdot A^{-1}\right)^{720} \supset \left\{ \left( \begin{pmatrix}  
    1 & Q_1\mathbb{Z}/q_1\mathbb{Z} \\ 0 & 1  
\end{pmatrix}, \begin{pmatrix}  
    1 & Q_2\mathbb{Z}/q_2\mathbb{Z} \\ 0 & 1  
\end{pmatrix}, \begin{pmatrix}  
    1 & Q_3\mathbb{Z}/q_3\mathbb{Z} \\ 0 & 1  
\end{pmatrix} \right) \right\} 
$$
Similarly, we can obtain $Q_1^{\prime}\lvert q_1$,$Q_2^{\prime}\lvert q_2$,$Q_3^{\prime}\lvert q_3$,$Q_1^{\prime}Q_2^{\prime}Q_3^{\prime}<(q_1q_2q_3)^{48\delta}$ and
$$
\left(A \cdot A^{-1}\right)^{720} \supset \left\{ \left( \begin{pmatrix}  
    1 & 0 \\ Q_1^{\prime}\mathbb{Z}/q_1\mathbb{Z} & 1  
\end{pmatrix}, \begin{pmatrix}  
    1 & 0 \\ Q_2^{\prime}\mathbb{Z}/q_2\mathbb{Z} & 1  
\end{pmatrix}, \begin{pmatrix}  
    1 & 0 \\ Q_3^{\prime}\mathbb{Z}/q_3\mathbb{Z} & 1  
\end{pmatrix} \right) \right\} 
$$
Let $Q_1^{*}=lcm(Q_1,Q_1^{\prime})$,$Q_2^{*}=lcm(Q_2,Q_2^{\prime})$,$Q_3^{*}=lcm(Q_3,Q_3^{\prime})$. It is easy to check that for $m<n$, any element of the group
\[
\begin{pmatrix}
a & b \\ c & d
\end{pmatrix} \supset SL_2(Z):
a,d \equiv 0 (\bmod p^{min\{2m,n\}}),b,c \equiv 0 (\bmod p^{min\{2m,n\}})
\]
can be written as $a_1b_1a_2b_2a_3$, where $a_1,a_2,a_3 \in \begin{pmatrix}
    1 & p^{m}Z/p^{n}Z \\0 & 1
\end{pmatrix}$ and $b_1,b_2 \in \begin{pmatrix}
    1 & 0 \\p^{m}Z/p^{n}Z & 1
\end{pmatrix}$. From this it follows that if we let $q_1^{\prime}=\operatorname{gcd}((Q_1^{*})^2,q_1)$,$q_2^{\prime}=\operatorname{gcd}((Q_2^{*})^2,q_2)$,$q_3^{\prime}=\operatorname{gcd}((Q_3^{*})^2,q_3)$then
\[
(A\cdot A^{-1})^{3600}\supset \Lambda(q_1^{\prime})/\Lambda(q_1)\times \Lambda(q_2^{\prime})/\Lambda(q_2)\times \Lambda(q_3^{\prime})/\Lambda(q_3)
\]
\end{proof}

\section{Random walks}
Given a finitely symmetric set $S$ on $\Gamma$ such that $\langle S \rangle$ is Zariski-dense, and $\chi_S$ is the uniform probability measure supported on S, we proposed the following propositions, which provide a quantitative version of non-concentration of self-convolutions of $\chi_S$in proper sub-varieties. Before the discussion begins, we preface a primitive linear form.
\begin{align*}
&L\left(\begin{pmatrix}
    x_1 & y_1 \\ z_1 & w_1
\end{pmatrix},\begin{pmatrix}
    x_2 & y_2 \\ z_2 & w_2
\end{pmatrix},\begin{pmatrix}
    x_3 & y_3 \\ z_3 & w_3
\end{pmatrix}\right)=X_1x_1+Y_1y_1+Z_1z_1+W_1w_1\\&+X_2x_2+Y_2y_2+Z_2z_2+W_2w_2+X_3x_3+Y_3y_3+Z_3z_3+W_3w_3    
\end{align*}
i.e.,$gcd(X_1,Y_1,Z_1,W_1,X_2,Y_2,Z_2,W_2,X_3,Y_3,Z_3,W_3)=1$
\begin{proposition}
Let S is a finitely symmetric set on $\Gamma$ such that $\langle S \rangle$ is Zariski-dense, and $\chi_S$ is the uniform probability measure supported on S. Then we can have a constant $c>0$ such that for $Q \in Z_+$, for any $l>logQ$ and $n \in Z$, we have
\[
\pi_Q^{*}[\chi_S^{(l)}](\{g\in SL_2(\mathbb{Z})\times SL_2(\mathbb{Z})\times SL_2(\mathbb{Z})|L(g)\equiv n(modQ)\})<Q^{-c}
\]    
\end{proposition}
\begin{proposition}
There is a constant $c$ such that the following holds. Let $Q \in \mathbb{Z}_{+}$large enough and $\xi=\left(\xi_1, \xi_2,\xi_3\right), \eta=\left(\eta_1, \eta_2,\eta_3\right) \gamma=\left(\gamma_1, \gamma_2\right)\in \operatorname{Mat}_2(\mathbb{Z}) \times \operatorname{Mat}_2(\mathbb{Z}) \times \operatorname{Mat}_2(\mathbb{Z})$ satisfy
\begin{equation}
\operatorname{Tr}\left(\xi_1\right)=\operatorname{Tr}\left(\xi_2\right)=\operatorname{Tr}\left(\xi_3\right)=\operatorname{Tr}\left(\eta_1\right)=\operatorname{Tr}\left(\eta_2\right)=\operatorname{Tr}\left(\eta_3\right)=0
\end{equation}
\begin{equation}
\pi_p\left(\xi_1\right), \pi_p\left(\xi_2\right),, \pi_p\left(\xi_3\right),\pi_p\left(\eta_1\right), \pi_p\left(\eta_2\right),\pi_p\left(\eta_3\right) \neq 0 \text { for every } p \mid Q,
\end{equation}
Then for $l>\log Q$,and Consider
$$
\kappa_i(g)=\operatorname{Tr}\left(g \xi_i g^{-1} \eta_i\right), g \in SL_2(\mathbb{Z}),i=1,2,3
$$
$$
\chi_S^{(l)}\left(\left\{\left(g_1, g_2,g_3\right) \in SL_2(\mathbb{Z}) \mid \kappa_1(g_1)+\kappa_2(g_2)+\kappa_3(g_3) \equiv 0(\bmod Q)\right\}\right)<Q^{-c}.
$$
\end{proposition}
This two proposition above have similar proof way,we have only focus on the proposition 4.1.
In order to prove this proposition, from the reference of \cite{tang2023super}, we have to prove the following lemma firstly:
\begin{lemma}
There are constants $c_1,c_2>0$ depending only on S such that for any $Q \in Z_+$,$n \in Z$ and for $l \in Z_+$ with $1\ll_S l<c_1logQ$, we have
\begin{equation}
\chi_S^{(l)}(\{g\in SL_2(\mathbb{Z})\times SL_2(\mathbb{Z})\times SL_2(\mathbb{Z})|L(g)\equiv n(modQ)\})<e^{-c_2{l}}    
\end{equation}
\end{lemma}
To prove this lemma 4.5, the following consequence should be considered
\begin{lemma}
There is a constant $c = c(S) > 0$ such that for any $l\gg_S 1$, we have
\[
\chi_S^{(l)}(\{g\in SL_2(\mathbb{Z})\times SL_2(\mathbb{Z})\times SL_2(\mathbb{Z})|L(g)=0\})<e^{-cl}
\]   
\end{lemma}
\begin{proof}
By Theorem 1.2, there exists an absolute constant $0<\lambda<1$ which is the upper bound for all eigenvalues of the family of the operator $T_p: l_0^{2}(\Gamma) \longrightarrow l_0^{2}(\Gamma)$ defined at (2.1)\\
So
\begin{equation}
\left\|\chi_S^{(l)}-\frac{1}{\left|\pi_p(\Gamma)\right|} \mathbf{1}_{\pi_p(\Gamma)}\right\|_2 \leq \lambda^l    
\end{equation}
for all prime $p$.Then since $|\pi_p(\Gamma)|\approx p^9$, if $l>\frac{9logp}{(log\frac{1}{\lambda})}$,
\begin{equation}
\chi_S^{(l)}(g)<\frac{3}{|\pi_p(\Gamma)|}
\end{equation}
for any $g\in \pi_p(\Gamma)$. Therefore,
\begin{align*}
&\chi_S^{(l)}(\{g\in SL_2(\mathbb{Z})\times SL_2(\mathbb{Z})\times SL_2(\mathbb{Z})|L(g)=n\})
\\<&\chi_S^{(l)}(\{g\in SL_2(\mathbb{Z})\times SL_2(\mathbb{Z})\times SL_2(\mathbb{Z})|L(g)\equiv n\bmod p)\})\\<&\frac{3p^2|SL_2(F_P)\times SL_2(F_P)|}{p^9}\\<&\frac{4}{p}      
\end{align*}

so we can take the prime $p\in [e^{\frac{l}{11}log\frac{1}{\lambda}},e^{\frac{l}{10}log\frac{1}{\lambda}}]$. This can be done when $l$ is sufficiently large.\\
\end{proof}
Now we can finish the proof of the lemma, we need the Effective Bezout Identity proposed in \cite{BY91}
\begin{proof}
Write $Q=\prod_{i\in I}p_i^{n_i}$. There are two cases: $n=0$ and $n\neq 0$. Considering $n=0$, at first, by the primitive of $L$, for each $p|Q$, at least one of
\[
t \in A:=\{X_1,Y_1,Z_1,W_1,X_2,Y_2,Z_2,W_2,X_3,Y_3,Z_3,W_3\}
\]
must be inevitable $\bmod p$. For each $t \in A$,let 
\begin{equation}
Q_t=\prod_{\substack{p_i \mid Q \\ \operatorname{gcd}\left(p_i, t\right)=1}} p_i^{n_i}
\end{equation}
Since $\prod_{t\in A}Q^{t}\geq Q$, there exists $t \in A$ such that 
\[
Q^{'}:=Q_t \geq Q^{\frac{1}{12}}
\]
Without loss of generality, we take $t=X_1$ so that $Q^{'}||Q$ and $(Q^{'},X_1)=1$.
Now we define the 
\begin{equation}
\left\|\left(g_1, g_2\right)\right\|=\max \left\{\text { absolute values of coefficients of } g_1 \text { and } g_2\right\} \text {. }
\end{equation}
Let the upper bound of $\left\| g \right\|$ for all $g \in supp\chi_S$ is $C_1$. Define 
\begin{equation}
\mathcal{G}=\left\{g \in \prod_l \operatorname{supp}\left[\chi_S\right] \mid L(g) \equiv 0(\bmod Q^{\prime})\right\}
\end{equation}.
the question is transfer to show:
\[
\chi_S^{(l)}(\mathcal{G})<e^{-cl}
\]
For each $\gamma=\begin{pmatrix}
    x_1 & y_1 \\ z_1 & w_1
\end{pmatrix},\begin{pmatrix}
    x_2 & y_2 \\ z_2 & w_2
\end{pmatrix},\begin{pmatrix}
    x_3 & y_3 \\ z_3 & w_3
\end{pmatrix}\in \mathcal{G}$ we have a linear polynomial
\begin{align*}
f_{\gamma}=&f_\gamma(\tilde{Y}_1, \tilde{Z}_1, \tilde{W}_1, \tilde{X}_2, \tilde{Y}_2, \tilde{Z}_2, \tilde{W}_2, \tilde{X}_3, \tilde{Y}_3, \tilde{Z}_3, \tilde{W}_3)\\
&\in Q(\tilde{Y}_1, \tilde{Z}_1, \tilde{W}_1, \tilde{X}_2, \tilde{Y}_2, \tilde{Z}_2, \tilde{W}_2,\tilde{X}_3, \tilde{Y}_3, \tilde{Z}_3, \tilde{W}_3),
\end{align*}
that is,
\begin{equation}
\begin{aligned}
&f_\gamma
=x_1+\tilde{Y}_1 y_1+\tilde{Z}_1 z_1+\tilde{W}_1 w_1\\
&+\tilde{X}_2 x_2+\tilde{Y}_2 y_2+\tilde{Z}_2 z_2+\tilde{W}_2 w_2+\tilde{X}_3 x_3+\tilde{Y}_3 y_3+\tilde{Z}_3 z_3+\tilde{W}_3 w_3
\end{aligned}
\end{equation}
Then we get
\begin{equation}
\begin{gathered}
X_1f_\gamma\equiv 0(\bmod Q^{'})\text{for all $\gamma \in \mathcal{G}$} 
\end{gathered}
\end{equation}
Here $\tilde{X}_1$ is the multiplicative inverse of $X_1 \bmod Q^{'}$.\\
Also, from the definition of $\mathcal{G}$, we can name the coefficient of $f_\gamma$ as $\gamma$ and find the following inequality 
$$
h(f_\gamma)<llogC_1
$$
From the condition we proposed above, we can claim there is a common zero $(\tilde{Y}, \tilde{Z}, \tilde{W}, \tilde{A}, \tilde{B}, \tilde{C}, \tilde{D},\tilde{E}, \tilde{F}, \tilde{G}, \tilde{H},\tilde{I})\in C^{11}$ to the following system of equations:
\begin{equation}
\begin{gathered}
f_\gamma(\tilde{Y}, \tilde{Z}, \tilde{W}, \tilde{A}, \tilde{B}, \tilde{C}, \tilde{D},\tilde{E}, \tilde{F}, \tilde{G}, \tilde{H})=0\\\text{for all $\gamma \in \mathcal{G}$} 
\end{gathered}    
\end{equation}
so that we lift the problem to $C$ by showing that $G$ is contained in some proper sub-variety of $SL_2(Z)\times SL_2(Z)\times SL_2(Z)$. Note that  is essentially $N\leq (2k)^l$ linear polynomials $F_1, . . . , F_N$.
Now we can invoke Theorem $n=11,d=3,h=llogC_1$,then there is an integer $M \in Z_+$ and polynomials $\varphi_1,...\varphi_N \in Z[X,Y]$ of degree at most $b=11\times23\times3^{11}$ satisfying 
\begin{equation}
M=\sum_{l=1}^N F_l\varphi_l    
\end{equation}
with
\begin{equation}
0<\log M, h\left(\varphi_l\right)<\mathfrak{X}(11) 3^{91}\left(\left(\log C_1+\log (|S|)\right) l+3 \log 3\right)<C_S^{\prime} l  \end{equation}
where
$C_S^{\prime}=3^{92} \mathfrak{X}(11) \log C_1|S|$.
Now we take
$$
\left(Y^{\prime}, Z^{\prime}, W^{\prime}, A^{\prime}, B^{\prime}, C^{\prime}, D^{\prime},E^{\prime},F^{\prime},G^{\prime},H^{\prime}\right) \in \mathbb{Z}^{11}
$$
and
$$
\iota=\left(Y_1 \bar{X}_1, Z_1 \bar{X}_1, W_1 \bar{X}_1, X_2 \bar{X}_1, Y_2 \bar{X}_1, Z_2 \bar{X}_1, W_2 \bar{X}_1,X_3 \bar{X}_1, Y_3 \bar{X}_1, Z_3 \bar{X}_1, W_3 \bar{X}_1\right)
$$
such that
$$
\left(Y^{\prime}, Z^{\prime}, W^{\prime}, A^{\prime}, B^{\prime}, C^{\prime}, D^{\prime},E^{\prime},F^{\prime},G^{\prime},H^{\prime}\right) \equiv\iota\left(\bmod Q^{\prime}\right) 
$$
It follows that
$$
\left.M \equiv \sum_{l=1}^N F_l \varphi_l\right|_{\left({Y^{\prime}, Z^{\prime}, W^{\prime}, A^{\prime}, B^{\prime}, C^{\prime}, D^{\prime},E^{\prime},F^{\prime},G^{\prime},H^{\prime}} \right)} \equiv 0\left(\bmod Q^{\prime}\right),
$$

Therefore, since $M \neq 0$, by  we deduce $\frac{1}{12} \log Q \leq \log Q^{\prime} \leq \log M<C_S^{\prime} l$, which contradicts the restriction $l<c_1 \log Q$ by taking $c_1=\frac{1}{12 C_S^{\prime}}>0$. This proves the claim. Since the linear system admits a solution and the coefficients of $f_\gamma$ are all integral, it must admit a rational solution 
$$
\zeta=(Y^{\prime}, Z^{\prime}, W^{\prime}, A^{\prime}, B^{\prime}, C^{\prime}, D^{\prime},E^{\prime},F^{\prime},G^{\prime},H^{\prime})
$$. 
In other words,
$$
x_1+\tilde{Y}_1 y_1+\tilde{Z}_1 z_1+\tilde{W}_1 w_1+\tilde{X}_2 x_2+\tilde{Y}_2 y_2+\tilde{Z}_2 z_2+\tilde{W}_2 w_2+\tilde{X}_3 x_3+\tilde{Y}_3 y_3+\tilde{Z}_3 z_3+\tilde{W}_3 w_3=0
$$
From the discussion above, It's easy for us to obtain $\zeta$ so that $$\operatorname{gcd}(X_1,Y_1,Z_1,W_1,X_2,Y_2,Z_2,W_2,X_3,Y_3,Z_3,W_3)=1
$$
and for all $\left(\begin{pmatrix}
    x_1 & y_1 \\ z_1 & w_1
\end{pmatrix},\begin{pmatrix}
    x_2 & y_2 \\ z_2 & w_2
\end{pmatrix},\begin{pmatrix}
    x_3 & y_3 \\ z_3 & w_3
\end{pmatrix}\right)\in \mathcal{G}$, there is 
\[
x_1+\tilde{Y}_1 y_1+\tilde{Z}_1 z_1+\tilde{W}_1 w_1+\tilde{X}_2 x_2+\tilde{Y}_2 y_2+\tilde{Z}_2 z_2+\tilde{W}_2 w_2+\tilde{X}_3 x_3+\tilde{Y}_3 y_3+\tilde{Z}_3 z_3+\tilde{W}_3 w_3=0
\]
The case $n\neq 0$ is  simpler, we just need to redefine
\begin{align*}
&f_\gamma\left(\tilde{Y}_1, \tilde{Z}_1, \tilde{W}_1, \tilde{X}_2, \tilde{Y}_2, \tilde{Z}_2, \tilde{W}_2, \tilde{X}_3, \tilde{Y}_3, \tilde{Z}_3, \tilde{W}_3\right)+n\\&=\tilde{X}_1x_1+\tilde{Y}_1 y_1+\tilde{Z}_1 z_1+\tilde{W}_1 w_1+\tilde{X}_2 x_2+\tilde{Y}_2 y_2+\tilde{Z}_2 z_2+\tilde{W}_2 w_2+\tilde{X}_3 x_3+\tilde{Y}_3 y_3\\&+\tilde{Z}_3 z_3+\tilde{W}_3 w_3    
\end{align*}
and proceed the analogous analysis as $n=0$.

Now we can finish the proof of lemma that the linear form determined by the constant we obtain from above discussion,
for $1\ll_S l<c_1logQ$,
$$
\begin{aligned}
& \chi_S^{(l)}\left(\left\{g \in \mathrm{SL}_2(\mathbb{Z}) \times \mathrm{SL}_2(\mathbb{Z}) \times \mathrm{SL}_2(\mathbb{Z})\mid L(g) \equiv 0(\bmod Q)\right\}\right) \\
\leq & \chi_S^{(l)}\left(\left\{g \in \mathrm{SL}_2(\mathbb{Z}) \times \mathrm{SL}_2(\mathbb{Z}) \times \mathrm{SL}_2(\mathbb{Z}) \mid L(g) \equiv 0\left(\bmod Q^{\prime}\right)\right\}\right) \\
= & \chi_S^{(l)}(\mathcal{G}) \\
\leq & \chi_S^{(l)}\left(\left\{\gamma \in \mathrm{SL}_2(\mathbb{Z}) \times \mathrm{SL}_2(\mathbb{Z}) \times \mathrm{SL}_2(\mathbb{Z})\mid \tilde{L}(\gamma)=0\right\}\right) \\
< & e^{-c_2 l} .
\end{aligned}
$$
where $c_2$ is the constant $c$ given in lemma.
\end{proof}
\begin{proof}{\text{of Proposition 4.1}}
Let $c_1, c_2$ be the constants given by Lemma 4.
Let $l_0=\left[c_1 \log Q\right]$ and write $\chi_S^{(l)}=\chi_S^{l_0} * \chi_S^{\left(l-l_0\right)}$. For any $g^{\prime}$ in the support of $\chi_S^{l-l_0}$, Consider $L_{g^{\prime}}(g)=L\left(g g^{\prime}\right)$. obviously, $L_{g^{\prime}}$ is also primitive, so Lemma is applicable to $L_{g^{\prime}}$.Therefore,
\begin{equation}
\begin{aligned}
& \pi_Q^*\left[\chi_S^{(l)}\right]\left(\left\{g \in \mathrm{SL}_2(\mathbb{Z}) \times \mathrm{SL}_2(\mathbb{Z}) \times \mathrm{SL}_2(\mathbb{Z}) \mid L(g) \equiv n(\bmod Q)\right\}\right) \\
=& \sum_{g^{\prime} \in \Gamma} \pi_Q^*\left[\chi_S^{\left(l_0\right)}\right]\left(\left\{g \in \Gamma\mid L_{\left(g^{\prime}\right)^{-1}}(g) \equiv n(\bmod Q)\right\}\right) \chi_S^{\left(l-l_0\right)}\left(g^{\prime}\right) \\
\leq & Q^{-c_1 c_2}
\end{aligned}
\end{equation}
Proposition 4.1 is proved by taking $c=c_1 c_2$.
\end{proof}
\section{Bounded generation}
Let $q=\prod_{i \in I} p_i^{n_i}, q_s=\prod_{i \in I: n_i \leq L} p_i^{n_i}, q_l=\prod_{i \in I: n_i>L} p_i^{n_i}$ for some $L$ to be specified later. In this section we assume
$$
q_l>q^{\epsilon / 2} \text {.}
$$

We let $c_0=c_0(L)$ be the implied constant from Theorem $\mathbf{1.1}$ for the power bound $L$. We fix $c_1, c_2$ to be the implied constant $c$ from Proposition $\mathbf{4.1}$ and Proposition $\mathbf{4.2}$.Assuming all the assumptions in Proposition $\mathbf{2.2}$ are satisfied but the conclusion fails, i.e.,
\begin{equation}
\left|\pi_q(A \cdot A \cdot A)\right| \leq\left|\pi_q(A)\right|^{1+\delta},    
\end{equation}
Take $A_0=A \cdot A \cap \Gamma\left(q_0\right)$, where $q_0=\prod_{p \mid q_l} p$. Then
$$
\chi_S^{2 l}\left(A_0\right)>q^{-2 \delta}/|SL_2(q_0)\times SL_2(q_0)\times SL_2(q_0)|\thickapprox q^{-2 \delta} q_0^{-18}>q^{-3 \delta}
$$
if we let
$$
L>\frac{18}{\delta} .
$$
With the above condition,we only have to make a minor change of proposition 5.6 of \cite{tang2023super}
\begin{proposition}
There are constants $c>0$ depending only on the generating set $S$ and $\epsilon$, in particular, independent of $\delta$, and $\rho=\rho(\delta)>0, C=C(\delta) \in \mathbb{Z}_{+}$, with $\rho(\delta) \rightarrow 0$ as $\delta \rightarrow 0$, such that
$$
\Lambda\left(\left(q^{\prime}\right)^{[\rho]}\right) / \Lambda\left(q^{\prime}\right) \subset\left(\mathbb{P}_i A_0\right)^C, i=1,2,3,
$$
where $q^{\prime} \| q_l$ for $L>\frac{18}{\delta}$ and $q^{\prime} \geq q_l^c$.    
\end{proposition}
We can straightly use the (5.86) of \cite{tang2023super}
\begin{equation}
c=\frac{c_1^2 c_3^2}{1024^2 C_1^2}, \rho=\frac{3 \times 10^{18} C_1^4 C_3^2}{c_1^9 c_2 c_3^5} \frac{\delta}{\epsilon}, C=C_3 2^{C_2} 8^{\left[\frac{C_1^5 C_3^3}{c_1^1 c_2 c_3^5} \frac{\epsilon}{\delta}\right]} .
\end{equation}
\section{Gluing moduli}
This part mainly use the the proposition 3.1 proposed in \cite{tang2023super}, we record it for reader's convenience.
\begin{proposition}
Let $G_1, G_2$ be two finite multiplicative groups and let $\psi: G_1 \rightarrow G_2$ some map. Then for $0<\varepsilon<\frac{1}{1600}$ we have either
\begin{equation}
\left|\left\{(x, y) \in G_1 \times G_1 \mid \psi(x y)=\psi(x) \psi(y)\right\}\right|<(1-\varepsilon)\left|G_1\right|^2 \text {, or }   
\end{equation}
there exists a subset $S \subset G_1$ with $|S|>(1-\sqrt{\varepsilon})\left|G_1\right|$ and a group homomorphism $f: G_1 \rightarrow$ $G_2$ such that
\begin{equation}
\left.f\right|_S=\left.\psi\right|_S
\end{equation}
\end{proposition}
By the proposition above,we can have the following conclusion.
\begin{proposition}
Suppose A satisfies (2.3) but fails (2.4). Let $0<\theta<10^{-18}$, and suppose
\begin{equation}
\delta<\frac{c_2 \epsilon \theta}{2}     
\end{equation}
where $\delta, \epsilon$ be given as in .
Let $q_1, q_2,q_3\left\|q, q_4\right\| q_l, \operatorname{gcd}\left(q_1, q_4\right)=1$, and $q_4>q^{288 \theta^{\frac{1}{2}}}$, where in the definition of $q_l$ we require $L>\frac{18}{\delta}$. Suppose for some set $B \subset \Gamma\left(q_0\right)$ where $q_0=\prod_{p \mid q_l} p$, we have
\begin{equation}
\begin{gathered}
\left|\pi_{q_1, q_2,q_3}(B)\right|>\left(q_1 q_2q_3\right)^{3-\theta} \\
\left|\pi_{q_4, 1,1}(B)\right|>q_4^{3-\theta}
\end{gathered}    
\end{equation}
Then there exists $q_4^* \| q_4, q_4^*>q_4^{\frac{1}{4} 10^{-4}}$, such that
\begin{equation}
\left|\pi_{q_1 q_4^*, q_2,q_3}(B \cup A)^{\left[200 \cdot 8^{\theta^{-\frac{1}{2}}}\right]}\right|>\left(q_1 q_2q_3q_4^*\right)^{3-900 \theta^{\frac{1}{4}}} .
\end{equation}
\end{proposition}

Write 
$$
\Lambda_{q_1q_4} \times \Lambda_{q_2} \times \Lambda_{q_3} \cong \Lambda_{q_1} \times \Lambda_{q_2} \times \Lambda_{q_3} \times \Lambda_{q_4} \times 1 \times 1
$$

Since $B$ satisfies (6.3) and (6.4), by Proposition 3.1, there exists $q_1^{\prime}\left|q_1, q_2^{\prime}\right| q_2, q_3^{\prime} \mid q_3, q_4^{\prime} \mid q_4,q_1^{\prime} q_2^{\prime}q_3^{\prime}<$ $\left(q_1 q_2q_3\right)^{96 \theta}, q_4^{\prime}<\left(q_4\right)^{96 \theta}$, such that
\begin{equation}
\Lambda\left(q_1^{\prime}\right) / \Lambda\left(q_1\right) \times \Lambda\left(q_2^{\prime}\right) / \Lambda\left(q_2\right) \times \Lambda\left(q_3^{\prime}\right) / \Lambda\left(q_3\right)\supset \pi_{q_1, q_2,q_3}\left(B^{7200}\right) .
\end{equation}
and
\begin{equation}
\Lambda\left(q_4^{\prime}\right) / \Lambda\left(q_4\right) \supset \pi_{q_4, 1,1}\left(B^{7200}\right) .
\end{equation}

Write
$$
G=\Lambda\left(q_1^{\prime}\right) / \Lambda\left(q_1\right) \times \Lambda\left(q_2^{\prime}\right) / \Lambda\left(q_2\right) \times \Lambda\left(q_3^{\prime}\right) / \Lambda\left(q_3\right)
$$

From (6.7), we can construct a map
$$
\psi: G \rightarrow B^{7200},
$$
such that
\begin{equation}
\pi_{q_1, q_2,q_3}(\psi(x))=x .
\end{equation}

From $q_4^{\prime}<q_4^{96\theta}$, there is $q_5 \| q_4, q_5>q_4^{1-\theta^{\frac{1}{2}}}$ such that for every $p^n \| q_5$, we have $p^{\left[96\theta^{\frac{1}{2}} n\right]} \nmid q_4^{\prime}$.\\
Let $q_5=\prod_{j \in J} p_j^{n_j}$. For each $p_j^{n_j} \| q_5$, we consider $\psi_j=\pi_{p_j^{n_j\theta^{\frac{1}{4}}}} \circ \mathbb{P}_1 \circ \psi$.\\
Let
\begin{equation}
\mathcal{G}_j=\left\{(x, y) \in G \times G \mid \psi_j(x y) \neq \psi_j(x) \psi_j(y)\right\}
\end{equation}

According to Proposition 6.1, there are two scenarios:\\
(1)
$$
\left|\mathcal{G}_j\right|>10^{-4}|G|^2 \text {. }
$$
(2) There is a subset $S_j \in G,\left|S_j\right| \geq \frac{99}{100}|G|$ such that $\psi_j=h_j$ over $S_j$ where $h_j$ is a homomorphism from $G$ to $\Gamma / \Gamma({p_j^{n_j\theta^{\frac{1}{4}}}})$.

Let $J=J_1 \sqcup J_2$ where $J_1, J_2$ is the collection of indices falling into Case (1) and Case (2), respectively. Write $q_5=q_5^{\prime} q_5^{\prime \prime}$, where
$$
q_5^{\prime}=\prod_{j \in J_1} p_j^{n_j}, q_5^{\prime \prime}=\prod_{j \in J_2} p_j^{n_j} .
$$

We further write $q_5^{\prime \prime}=q_6q_6^{\prime}$, where

\begin{equation}
q_6^{\prime}=\prod_{\substack{p_j^{n_j} \| q_4^{\prime \prime} \\ \pi_{p_j^[\frac{n_j\theta^{1/4}}{2}]}}\circ p_j^{n_j}=1} {p_j^{n_j}}
\end{equation}
and 
\begin{equation}
q_6^{\prime\prime}=\prod_{\substack{p_j^{n_j} \| q_4^{\prime \prime} \\ \pi_{p_j^[\frac{n_j\theta^{1/4}}{2}]}}\circ p_j^{n_j}\neq1} {p_j^{n_j}}
\end{equation}

There are three cases can be analysed.
\subsection{The case $q_5^{\prime}>q_5^{\frac{1}{2}}$}
In this case,
\begin{equation}
\sum_{j \in J_1}\left(\log p_j^{n_j}\right)\left|\mathcal{G}_j\right|>\log \left(\left(q_5^{\prime}\right)^{10^{-4}}\right)|G|^2
\end{equation}

The left side of the inequality 6.14 is equal to
\begin{equation}
\sum_{\substack{U \subset J_1 \\ U \neq \emptyset}} \log \left(\prod_{j \in U} p_j^{n_j}\right)\left|\cap_{j \in U} \mathcal{G}_j \bigcap \cap_{j \in J_1-U} \mathcal{G}_j^c\right|
\end{equation}

Since the number of subsets of $J_1$ is $<q_5^{0+}$, we have $J_1^{\prime} \subset J_1, \tilde{q}=\prod_{i \in J_1} p_i^{n_i}>\left(q_5^{\prime}\right)^{\frac{1}{2} 10^{-4}}$, such that
$$
\left|\cap_{i \in J_1} \mathcal{G}_i\right|>q^{0-}|G|^2 .
$$

Take any $\left(g_1, g_2\right) \in \cap_{i \in J_1} \mathcal{G}_i$, and consider $\gamma_0=\psi\left(g_1\right) \psi\left(g_2\right) \psi\left(g_1 g_2\right)^{-1}$. Then $\gamma_0$ satisfies,
$$
\pi_{q_1, q_2,q_3}\left(\gamma_0\right)=1
$$
and for any $p^n|| \tilde{q}$,
$$
\pi_{p^{\left[n \theta \frac{1}{4}\right]}} \circ \mathbb{P}_1\left(\gamma_0\right) \neq 1
$$
From this two formula, we can have a series of congruence,$X_i \in V=\operatorname{Lie}\left(\mathrm{SL}_2\right)(\mathbb{Z})$.
\[
\mathbb{P}_1\left(\gamma_0\right) \equiv X_i (\bmod p^{\left[n \theta \frac{1}{4}\right]})
\]
$X_i \in V=\operatorname{Lie}\left(\mathrm{SL}_2\right)(\mathbb{Z})$
Then by the Chinese reminder theorem, we can easily write the $\gamma_0$ into the following form.
$$
\gamma_0 \equiv 1+\left(\prod_{p \mid q_4^{\prime}} p^{m_p}\right) X\left(\bmod \prod_{p \mid q_4^{\prime}} p^{2 m_p}\right),
$$
for some primitive $X \in V=\operatorname{Lie}\left(\mathrm{SL}_2\right)(\mathbb{Z})$ (i.e., the gcd of entries of $X$ is 1 ) and $1 \leq m_p \leq$ $\left[n \theta^{\frac{1}{4}}\right]$.\\
To proceed, we have the following trivial lemma
\begin{lemma}
Given $q \in \mathbb{Z}_{+}$and $\vec{v}, \vec{w} \in V$ primitive. Suppose for any $p \mid q, \vec{v}$ and $\vec{w}$ are linearly independent mod $p$. Then
$$
[v, V]+[w, V] \supset 2 V(\bmod q)
$$    
\end{lemma}
We take two elements $\gamma_1, \gamma_2 \in B^{7200}$ so that
$$
\mathbb{P}_1\left(\gamma_0 \gamma_1 \gamma_0^{-1} \gamma_1^{-1}\right) \equiv 1+\tilde{q}^{\left\{2 \theta^{\frac{1}{4}}\right\}} Y_1\left(\bmod \tilde{q}^{\left\{4 \theta^{\frac{1}{4}}\right\}}\right)
$$
and
$$
\mathbb{P}_1\left(\gamma_0 \gamma_2 \gamma_0^{-1} \gamma_2^{-1}\right) \equiv 1+\tilde{q}^{\left\{2 \theta^{\frac{1}{4}}\right\}} Y_2\left(\bmod \tilde{q}^{\left\{4 \theta^{\frac{1}{4}}\right\}}\right)
$$
for some $Y_1, Y_2$ satisfying the hypothesis of $\vec{v}, \vec{w}$ in Lemma 6.16 with $q$ replaced by $\tilde{q}$. Importantly,
\begin{equation}
\pi_{q_1, q_2,q_3}\left(\gamma_0 \gamma_1 \gamma_0^{-1} \gamma_1^{-1}\right)=\pi_{q_1, q_2,q_3}\left(\gamma_0 \gamma_2 \gamma_0^{-1} \gamma_2^{-1}\right)=1 .
\end{equation}

We also take
$$
H_\rho=\left\{\gamma \in B^{7200}, \pi_{\tilde{q}[\rho]}(\gamma)=1\right\}
$$
for $\rho \in\left[\theta^{\frac{1}{4}}, \frac{1}{2}\right]$.
From the condition we can easily find
$$
\pi_{\tilde{q}^{(2 \rho\}}} \mathbb{P}_1\left(B^{7200}\right) \supset \Lambda\left(\tilde{q}^{\{\rho\}}\right) / \Lambda\left(\tilde{q}^{\{2 \rho\}}\right) .
$$
Considering $\rho=\frac{1}{2}$,$\gamma \in H_{\rho}$,$\gamma_0 \gamma_1 \gamma_0^{-1} \gamma_1^{-1}=\vec{v}$,$\gamma_0 \gamma_2 \gamma_0^{-1} \gamma_2^{-1}=\vec{w}$,$\gamma \in H_{\rho}$ means $\gamma =1+\tilde{q}^{\rho}X_{\gamma}$ where $X_{\gamma} \in \operatorname{Lie}\left(\mathrm{SL}_2\right)(\mathbb{Z})$.Following the calculation of lemma 6.16, we have
\[
\tilde{q}^{\rho+2\theta^{\frac{1}{4}}}\left((Y_1X_r-X_rY_1)+(Y_2X_r-X_rY_2)\right) 
\]
It is easy to find that $\tilde{q}^{\rho}(Y_1X_r-X_rY_1) \equiv Y_1^{\prime}(\bmod \tilde{q}^{4\rho_1^{\frac{1}{4}}})$,$\tilde{q}^{\rho}(Y_2X_r-X_rY_2) \equiv Y_2^{\prime}(\bmod \tilde{q}^{4\rho_1^{\frac{1}{4}}})$,$Y_1^{\prime},Y_2^{\prime}\in \operatorname{Lie}\left(\mathrm{SL}_2\right)(\mathbb{Z})$.So we can get the result that
\[
\left[\vec(v),H_{\rho}\right]+\left[\vec(w),H_{\rho}\right]\subset \Lambda\left(\tilde{q}^{\left\{6 \theta^{\frac{1}{4}}\right\}}\right) \subset \Lambda(\tilde{q}^{\rho}) / \Lambda(\tilde{q}^{2\rho})
\]
Since $\rho=\frac{1}{2}$,after $\frac{1}{6\theta^{\frac{1}{4}}}$ times calculation,the consequence we obtain is below.
\begin{equation}
\begin{aligned}
& \pi_{q_1, q_2,q_3}(F)=1, \\
& \pi_{\bar{q}, 1,1}(F) \supset \Lambda\left(\tilde{q}^{\left\{3 \theta^{\frac{1}{4}}\right\}}\right) / \Lambda(\tilde{q})
\end{aligned}
\end{equation}

In this case, we take $q_4^*=\tilde{q}$, so
$$
q_4^*>q_4^{\frac{1}{4} 10^{-4}}
$$
and it follows from $(6.8),(6.18)$ that
$$
\left|\pi_{q_1 q_4^*, q_2,q_3}\left(B^{\left[\frac{43200}{\theta^{1 / 4}}\right]}\right)\right|>\left(q_1 q_2 q_3 q_4^*\right)^{3-9 \theta^{\frac{1}{4}}} .
$$
\subsection{The case $q_5^{\prime \prime}>q_5^{\frac{1}{2}}, q_6>\left(q_6^{\prime \prime}\right)^{\frac{1}{2}}$}
The local homomorphisms $h_j, j \in J_2$ can be lifted to a homomorphism
$$
h: G \rightarrow \Lambda\left(\prod_{j \in J_2} p_j\right) / \Lambda\left(\left(q_4^{\prime \prime}\right)^{\left\{\theta^{\frac{1}{4}}\right\}}\right) .
$$

Following the previous reasoning for obtaining $J_1^{\prime}$, there is a set $J_2^{\prime} \subset J_2$, such that
$$
\begin{gathered}
\bar{q}=\prod_{i \in J_2^{\prime}} p_i^{n_i}>\left(q_4^{\prime \prime}\right)^{\frac{99}{200}}, \\
\left|\cap_{i \in J_2} S_i\right|>\left(q_1 q_2\right)^{-\theta}|G|,
\end{gathered}
$$
and $\psi \equiv h$ on $S=\cap_{i \in J_2} S_i$.
By Proposition 3.1, we have $S^{7200} \supset G^{\prime}=\Lambda\left(q_1^{\prime \prime}\right) / \Lambda\left(q_1\right) \times \Lambda\left(q_2^{\prime \prime}\right) / \Lambda\left(q_2\right)\times \Lambda\left(q_3^{\prime \prime}\right) / \Lambda\left(q_3\right)$ for some $q_1^{\prime}\left|q_1, q_2^{\prime}\right| q_2, q_3^{\prime} \mid q_3, q_4^{\prime} \mid q_4,q_1^{\prime} q_2^{\prime}q_3^{\prime}<$ $\left(q_1 q_2q_3\right)^{200\theta}, q_4^{\prime}<\left(q_4\right)^{200 \theta}$. This implies the existence of a subgroup $G^{\prime \prime}$ of $G^{\prime}$ of the form
\begin{equation}
G^{\prime \prime}=\Lambda\left(\tilde{q}_1\left(q_1^*\right)^{\left\{\theta^{\frac{1}{2}}\right\}}\right) / \Lambda\left(q_1\right) \times \Lambda\left(\tilde{q}_2\left(q_2^*\right)^{\left\{\theta^{\frac{1}{2}}\right\}}\right) / \Lambda\left(q_2\right) \times \Lambda\left(\tilde{q}_3\left(q_3^*\right)^{\left\{\theta^{\frac{1}{2}}\right\}}\right) / \Lambda\left(q_3\right)
\end{equation}
where
\begin{equation}
q_1=\tilde{q}_1 q_1^*, q_2=\tilde{q}_2 q_2^*, q_3=\tilde{q}_3 q_3^*, \tilde{q}_1, q_1^*\left\|q_1, \tilde{q}_2, q_2^*\right\| q_2, \tilde{q}_3, q_3^*\| q_3,\tilde{q}_1 \tilde{q}_2\tilde{q}_3<\left(q_1 q_2q_3\right)^{200 \theta^{\frac{1}{2}}} .
\end{equation}

Defining a map $\tilde{\psi}: G^{\prime \prime} \rightarrow B^{7200^2}$ , for any $x \in G^{\prime \prime}$, choose a word $s_1 s_2 \cdots s_{7200}$ and let
\begin{equation}
\tilde{\psi}(x)=\psi\left(s_1\right) \psi\left(s_2\right) \cdots \psi\left(s_{7200}\right) .
\end{equation}

It is easy to find $\pi_{\left(q_5^{\prime \prime}\right)^{\left\{\theta^{\frac{1}{4}}\right\}}} \mathbb{P}_1 \tilde{\psi}(x)=h(x)$.
Since $q_6>\left(q_5^{\prime \prime}\right)^{\frac{1}{2}}$, and the commutator of $\Gamma(p^m)/\Gamma(p^n)$ is $\Gamma(p^{2m})/\Gamma(p^n)$ we can easily get the following result by taking commutator of$\tilde{\psi}\left(G^{\prime \prime}\right)\left[\frac{2}{\theta^{1 / 4}}\right]$ times, 
\begin{equation}
\begin{aligned}
&\Lambda\left(\tilde{q}_1\left(q_1^*\right)^{\left\{2 \theta^{\frac{1}{4}}\right\}} q_5\right) / \Lambda\left(q_1 q_5\right) \times \Lambda\left(\tilde{q}_2\left(q_2^*\right)^{\left\{2 \theta^{\frac{1}{4}}\right\}}\right) / \Lambda\left(q_2\right) \times \Lambda\left(\tilde{q}_3\left(q_3^*\right)^{\left\{2 \theta^{\frac{1}{4}}\right\}}\right) / \Lambda\left(q_3\right)\\& \supset B^{7200^2 \cdot\left(3 \cdot 2^{\left[\theta^{-\frac{1}{4}}\right]}-2\right)}
\end{aligned}
\end{equation}

Now We consider $q_4^*=q_6$. We have
$$
q_4^*>q_4^{\frac{1}{8}}
$$

It follows from (6.20) and (6.22) that
\begin{equation}
\left|\pi_{q_1 q_4^*, q_2,q_3}\left(B^{3\cdot7200^2 \cdot 2^{\left[\theta^{-\frac{1}{4}}\right]}}\right)\right|>\left(q_1 q_2 q_3q_4^*\right)^{3-900 \theta^{\frac{1}{4}}} .
\end{equation}
\subsection{The case $q_5^{\prime \prime}>q_5^{\frac{1}{2}}, q_6^{\prime}>\left(q_5^{\prime \prime}\right)^{\frac{1}{2}}$}
By the lemma 6.21 present in \cite{tang2023super},which we will present below for reference
\begin{lemma}
Suppose $h$ is a homomorphism from $\Lambda\left(p_1^{m_1}\right) / \Lambda\left(p_1^{n_1}\right)$ to $\Lambda\left(p_2^{m_2}\right) / \Lambda\left(p_2^{n_2}\right)$ for some $0 \leq m_1 \leq n_1, 1 \leq m_2 \leq n_2$, and for some $\xi \in \Lambda\left(p_1^{m_1}\right) / \Lambda\left(p_1^{n_1}\right), h(\xi) \neq 1$. Then $p_1=p_2$ and $p^x|| \xi-1$ for some $x \leq m_1+n_2-m_2$.    
\end{lemma}
Since $\operatorname{gcd}\left(q_1, q_4\right)=1$ ,$\pi_{\left(q_6^{\prime}\right)^{\left\{\theta^{\frac{1}{4}}\right\}}} \circ \mathbb{P}_1 \circ \tilde{\psi}$ is a homomorphism, and $\pi_{\left(q_6^{\prime}\right)^{\left\{\frac{\theta^{\frac{1}{4}}}{2}\right\}}} \circ \mathbb{P}_1 \circ \tilde{\psi}$ is nontrivial for each $p^n \| q_5^{\prime}$.Suppose $(q_1, q_4)\neq 1$ We can find that for each write $p^{t_1} \|\left(\mathbb{P}_1 \tilde{\psi}(g)-1\right)$ and $p^{t_2} \|\left(\mathbb{P}_2 \tilde{\psi}(g)-1\right)$. We have $\left[\frac{\theta^{\frac{1}{4}}}{2} n\right]<t_1<1$,fix $t_1=\left[\frac{2\theta^{\frac{1}{4}}}{3} n\right]$,then $\left[\frac{\theta^{\frac{1}{4}}}{6} n\right]<t_2<\left[\left(\frac{2 \theta^{\frac{1}{4}}}{3}+\theta^{1 / 2}\right) n\right]<\left[\left(\frac{3 \theta^{\frac{1}{4}}}{4}\right) n\right]$.
\begin{lemma}
There are $Q_1\left|\hat{Q_1}\right| q_1^* q_6^{\prime}, Q_2\left|\hat{Q_2}\right| q_2^*,Q_3\left|\hat{Q_3}\right| q_3^*q_6^{\prime}, \xi=\left(\xi_1, \xi_2,\xi_3\right) \in \operatorname{Lie}\left(S L_2\right)(\mathbb{Z}) \times \operatorname{Lie}\left(S L_2\right)(\mathbb{Z})\times \operatorname{Lie}\left(S L_2\right)(\mathbb{Z})$, $\xi_1, \xi_2,\xi_3$ are primitive, such that
$$
\pi_{\hat{Q_1},  \hat{Q_2}, \hat{Q_3}}\left[(1,1,1)+\mathbb{Z}\left(Q_1 \xi_1, Q_2 \xi_2, Q_3 \xi_3\right)\right] \subset \pi_{\hat{Q_1},  \hat{Q_2}, \hat{Q_3}}\left(B^{7200^2}\right),
$$
The values of $Q_1, Q_2,Q_3, \hat{Q_1},  \hat{Q_2}, \hat{Q_3}$ are specified as follows:\\
For each $p^n \| q_5^{\prime}$, write $p^{t_1}\left\|Q_1, p^{t_2}\| Q_2,p^{t_3}\right\| Q_3$, we have $t_3=\left[\frac{2 \theta^{\frac{1}{3}}}{3} n\right],t_2=\left[\frac{2 \theta^{\frac{1}{3}}}{3} n\right],\left[\frac{\theta^{\frac{1}{2}}}{6} n\right]<t_1<\left[\left(\frac{3 \theta^{\frac{1}{2}}}{4}\right) n\right]$, and $p^{2 t_1} \| \hat{Q_1}$ , $p^{\left[\frac{4 t_2}{3}\right]} \| \hat{Q_2}$ and $p^{\left[\frac{4 t_2}{3}\right]} \| \hat{Q_3}$.\\
For each $p^n \| q_2^*$ but $p \neq q_5^{\prime}$, we have $p^n \| Q_2, \hat{Q_2}$.\\
For each $p^n \| q_3^*$ but $p \neq q_5^{\prime}$, we have $p^n \| Q_3, \hat{Q_3}$.\\
For each $p^n \| q_1^*$, we have $p^n \| Q_1, \hat{Q_1}$.    
\end{lemma}

With Lemma 6.25 we have, which can create a one-parameter group, we follow the same procedures in \cite{tang2023super} of Section 5.5 and Section 5.6.The starting step is to apply Proposition to find elements from $A$ to conjugate $\left(\xi_1, \xi_2,\xi_3\right)$ to other directions. By considering eight primitive linear forms, we can produce $g_1, g_2, g_3, g_4, g_5,g_6,g_7,g_8 \in A$ such that the following holds:
$\overline{q_6} \| q_6$ and
\begin{equation}
\overline{q_6}>\left(q_6^{\prime}\right)^{\{1-\theta^{\frac{1}{2}}\}^8} >\left(q_6^{\prime}\right)^{1-8 \theta^{\frac{1}{2}}},
\end{equation}
such that
\begin{equation}  
\begin{aligned}  
&\left(\bar{q_6}\right)^{\{8 \theta^{\frac{1}{2}}}\} \pi_{\bar{q_6}}(V \times V\times V) \\& \subset \pi_{\bar{q}_5} \operatorname{Span}_{\mathbb{Z}}\left\{\xi, g_1 \xi g_1^{-1}, g_2 \xi g_2^{-1}, g_3 \xi g_3^{-1}, g_4 \xi g_4^{-1}, g_5 \xi g_5^{-1}, g_6 \xi g_6^{-1}, g_7 \xi g_7^{-1}, g_8 \xi g_8^{-1}\right\}  
\end{aligned}  
\end{equation}
as long as
\begin{equation}
\delta<\frac{c_2 \epsilon \theta}{2} .
\end{equation}

Let $\bar{Q}_1=\operatorname{gcd}\left(Q_1, \bar{q}_6\right), \bar{Q}_2=\operatorname{gcd}\left(Q_2, \bar{q}_6\right),\bar{Q}_3=\operatorname{gcd}\left(Q_3, \bar{q}_6\right)$. Take Lemma 6.25 and (6.27) into localized form and then generalized into $\bar{Q}_i,i=1,2,3$ we can easily determine the existence of $F_1 \subset\left\{B^{7200^2} \cup A\right\}^{10}$ such that
\begin{equation}
\pi_{\bar{Q}_1{ }^{\left\{\frac{4}{3}\right\}}, \bar{Q}_2^{\left\{\frac{4}{3}\right\}}, \bar{Q}_3^{\left\{\frac{4}{3}\right\}}\left[\left(1+{\overline{Q_1}}^{\left\{\frac{5}{4}\right\}} V\right),\left(1+{\overline{Q_2}}^{\left\{\frac{5}{4}\right\}} V\right),\left(1+{\overline{Q_3}}^{\left\{\frac{5}{4}\right\}} V\right)\right]} \subset F_1 
\end{equation}
and
\begin{equation}
\pi_{q_1^*, q_2^* / \overline{q_6},q_3^*}\left(F_1\right)=1 \text {. }
\end{equation}

Then taking commutator of the left hand side of (6.29) and taking further commutator iterative, we obtain a set $F_2 \subset\left\{B^{7200^2} \cup A\right\}^{\left[200 \cdot 8^{\theta^{-\frac{1}{4}}}\right]}$ such that
\begin{equation}
\Lambda\left({\overline{Q_1}}^{20}\right) / \Lambda\left(\bar{q}_6\right) \times \Lambda\left({\overline{Q_2}}^{20}\right) / \Lambda\left(\bar{q}_6\right) \times \Lambda\left({\overline{Q_3}}^{20}\right) / \Lambda\left(\bar{q}_6\right) \subset \pi_{\bar{q}_6}\left(F_2\right),
\end{equation}

and
\begin{equation}
\pi_{q_1^*, q_2^* / \bar{q}_6,q_3^*}\left(F_2\right)=1
\end{equation}

To solve this case, let $q_4^*=\bar{q}_6$, we have
\begin{equation}
q_1 q_2 q_3 q_4^*>\left(q_1^* q_2^* q_3^*q_6^{\prime}\right)^{1-8 \theta^{\frac{1}{2}}}>\left(q_1 q_2q_3\right)^{1-9 \theta^{\frac{1}{2}}} q_3^{\frac{1}{8}},
\end{equation}
which implies
\begin{equation}
q_4^*>q_4^{\frac{1}{8}} q^{-18 \theta^{\frac{1}{2}}}>q_3^{\frac{1}{16}} .
\end{equation}

if $q_1 q_2 q_3 / q_4^*>\left(q_1 q_2q_3\right)^{\theta^{\frac{1}{2}}}$, then
$$
\left|\pi_{q_1, q_2 / q_4^*,q_3}(B)\right|>\left(q_2 / q_4^*\right)^{3-3 \theta^{\frac{1}{2}}} .
$$
with (6.31) and (6.32), it's easy to have
\begin{equation}
\left|\pi_{q_1 q_4^*, q_2,q_3}\left(\left\{B^{7200^2} \cup A\right\}^{\left[200 \cdot 8^{\theta^{-\frac{1}{4}}}\right]}\right)\right|>\left(q_1 q_2q_3 q_4^*\right)^{3-60 \theta^{\frac{1}{4}}} .
\end{equation}

If $q_1 q_2 q_3 / q_4^*\leq\left(q_1 q_2\right)^{\theta^{\frac{1}{2}}}$, then $q_4^*>q_2^{1-\theta^{\frac{1}{2}}}$ and $q_1<q_2^{\theta^{\frac{1}{2}}}$. Then Proposition 6.4 is thus proved.\\

\section{Proof of Proposition 2.2}
Recall $q=q_s q_l$, where $q=\prod_{i \in I} p_i^{n_i}, q_s=\prod_{i \in I: n_i \leq L} p_i^{n_i}, q_l=\prod_{i \in I: n_i>L} p_i^{n_i}$ for some $L$ to be determined at (7.15). \\
We divide our proof into three cases.
\subsection{The case $q_l<q^{\frac{\epsilon}{2}}$} 
Take
\begin{equation}
\delta<\frac{\log \frac{1}{c_0}}{6}
\end{equation}
for $c_0=c_0(L)$ the implied constant from Theorem 1.2, so that
$$
\left|\pi_{q_s}^*\left(\chi_S^{(l)}\right)(x)-\frac{1}{\left|\Lambda_{q_s}\right|}\right|<\frac{1}{2}
$$
for any $x \in \Lambda_{q_s}$. Since $\pi_{q_s}^*\left(\chi_S^{(l)}\right)(A)>q^{-\delta}$, we have $|A|>q^{9-3 \delta}$ for $q$ sufficiently large. By taking
\begin{equation}
\delta=\min \left\{\frac{\log \frac{1}{c_0}}{6}, \frac{\epsilon}{3}\right\},
\end{equation}
we have $|A|>q^{1-\epsilon}$, so the assumption of Proposition 2.2  is void and Proposition 2.2 automatically holds.

\subsection{The case $q^{\frac{\epsilon}{2}}<q_l<q^{1-\frac{\epsilon}{2}}$}
We assume the conclusion of Proposition 2.2 fails, i.e.
\begin{equation}
\left|\pi_q(A \cdot A \cdot A)\right| \leq\left|\pi_q(A)\right|^{1+\delta} .
\end{equation}
and we will arrive at a contradiction if we take $\delta$ sufficiently small.
If
\begin{equation}
\delta \leq \frac{\log \frac{1}{c_0}}{6},
\end{equation}
we have
\begin{equation}
\left|\pi_{q_s}(A)\right|>q^{-\delta} q_s^9>q_s^{9-\frac{\delta}{\epsilon}},
\end{equation}

Here we can view $\frac{\delta}{\epsilon}$ measures the closeness of $\pi_{q_s}(A)$ to $\Lambda_{q_x}$.
With the proposition 5.1,we have some $q^{\prime} \| q_l, q^{\prime}>q_l^{\frac{c_1 c_3}{1024 L_1}}$, such that
$$
\left|\pi_{q^{\prime}} \circ \mathbb{P}_1\left(A^{C^{\prime}}\right)\right|>\left(q^{\prime}\right)^{3-\theta}
$$
where both $C^{\prime}$ and $\theta$ are functions of $\delta$ given by
\begin{equation}
C^{\prime}=C^{\prime}(\delta)=C_3 2^{C_2+1} 8^{\left[\frac{C_1^5 c_3^3}{c_{11} c_2 c_3^5} \frac{5}{z}\right]}
\end{equation}
\begin{equation}
\theta=\theta(\delta)=\frac{3 \times 10^{18} C_1^4 C_3^2}{c_1^9 c_2 c_3^5} \frac{\delta}{\epsilon} .
\end{equation}

We clearly have $\theta(\delta) \rightarrow 0$ as $\delta \rightarrow 0$.
Now we apply Proposition 6.2 with $B=A^{C^{\prime}}, q_1=q_2=q_3=q_s, q_4=q^{\prime}$ and $\theta$ given in (7.7), with the requirement that
\begin{equation}
\delta<\frac{c_1^9 c_2 c_3^5}{3 \times 10^{36} C_1^4 C_3^2} \epsilon,
\end{equation}
so that the assumption (6.4) in Proposition 7.6 is satisfied. We then obtain $q^{\prime \prime} \mid q^{\prime}, q^{\prime \prime}>$ $q^{\prime \frac{1}{4} 10^{-4}}$, and
\begin{equation}
\left|\pi_{q_s q^{\prime \prime}, q_s,q_s}\left(A^{C^{\prime} C^{\prime \prime}}\right)\right|>\left(q_s^3 q^{\prime \prime}\right)^{3-900 \theta^{\frac{1}{4}}}
\end{equation}
where
\begin{equation}
C^{\prime \prime}=\left[200 \cdot 8^{\theta^{-\frac{1}{2}}}\right]
\end{equation}

By (7.10), we have increased the modulus of the first component from $q_s$ to $q_s q^{\prime \prime}$, at the cost of a density loss from $3-\theta$ to $3-900 \theta^{\frac{1}{4}}$.

If $q_s q^{\prime \prime}<q^{1-\frac{\varepsilon}{2}}$, we apply Proposition 5.6 with modulus $q_l$ replaced by $\frac{q}{q_s q^{\prime \prime}}$, then we apply Proposition 6.4 to increase the modulus to a larger one. Applying Proposition 5.1 and Proposition 6.4 repeatably until we reach a modulus $q_1^{\star}|| q, q_1^{\star}>q^{1-\frac{\varepsilon}{2}}$. Next, we go through
the same procedure to increase the modulus of the second component to $q_2^{\star} \| q, q_2^{\star}>q^{1-\frac{\epsilon}{2}}$.Finally follow the same steps to increase the modulus of the third component to $q_3^{\star} \| q, q_3^{\star}>q^{1-\frac{\epsilon}{2}}$. In total it takes at most $\left[\frac{10^{12} C_1}{c_1 c_3 \epsilon}\right]$ steps. In the end, recall also $\pi_q(A)<q^{9-9\epsilon}$, so
\begin{equation}
\left|\pi_q\left(A^{C^{\prime}\left(C^{\prime \prime}\right)}{ }^{\left[\frac{10^{12} C_1}{c_1 c_3 \epsilon}\right]}\right)\right|>|A|^{1+\frac{\epsilon}{3}}
\end{equation}
if
\begin{equation}
3\left(900 \theta^{\frac{1}{4}}\right)^{\left[\frac{10^{12} C_1}{c_1 c^\epsilon}\right]}<\epsilon
\end{equation}

Then (7.12) and (5.2) imply

We take $\delta_0$ sufficiently small so that $\theta=\theta\left(\delta_0\right)$ satisfies (7.13). Then we take
\begin{equation}
L=\left[\frac{20}{\delta_0}\right]
\end{equation}
Consider $C^{\prime}=C^{\prime}\left(\delta_0\right)$ given at (7.6) and $C^{\prime \prime}=C^{\prime \prime}\left(\theta\left(\delta_0\right)\right)$given at (7.10). Finally, we set
$$
\delta=\min \left\{\delta_0, \frac{\log \frac{1}{c_0}}{6}, \frac{\epsilon}{3 C^{\prime}}\left(C^{\prime \prime}\right)^{-\left[\frac{10^{12} C_1}{c_1 c_3{ }^e}\right]}\right\}
$$

Then we have contradiction.That's prove the correctness of Proposition 2.2.
\subsection{The case $q_l>q^{1-\frac{\epsilon}{2}}$}
In this case the modulus $q_s$ can be ignored. The Proposition 5.1 can be used to create a product set $B$ of $A$ such that $\pi_{q_1} \mathbb{P}_1(B)$ and $\pi_{q_2} \mathbb{P}_1(B)$ and $\pi_{q_3} \mathbb{P}_1(B)$ are large, where $q_1, q_2,q_3 \| q_l, \operatorname{gcd}\left(q_1, q_2\right)=1$.By the assumption, the analysis is virtually identical to the previous case.
Proposition 2.2 is thus fully proved.

\printbibliography

\end{document}